\documentclass[11pt]{amsart}
\usepackage{amsfonts,amssymb,amsmath,amsthm}
\usepackage{url}
\usepackage{enumerate}
\usepackage[francais]{}
 \usepackage{hyperref}
\usepackage{fancyhdr}
\newtheorem{theorem}{Theorem}[section]
\newtheorem{lemma}[theorem]{Lemma}
\newtheorem{corollary}[theorem]{Corollary}
\newtheorem{conjecture}[theorem]{Conjecture}

\newtheorem{example}[theorem]{Example}

\newtheorem{definition}[theorem]{Definition}

\newtheorem{remark}[theorem]{Remark}
\newtheorem{remarks}[theorem]{Remarks}
\numberwithin{equation}{section}
\def\wt{\widetilde}
\def\notdiv{\nmid}
\def\div{\,\vert\,}

\def\Q{\mathbb{Q}}
\def\Z{\mathbb{Z}}

\def\ov{\overline}

\def\Cl{{\mathcal C}\hskip-2pt{\ell}}

\def\Frac#1#2{\hbox{\footnotesize $\displaystyle \frac{#1}{#2}$}}
\def\plus{\displaystyle\mathop{\raise 2.0pt \hbox{$\bigoplus $}}\limits}
\def\prd{ \displaystyle\mathop{\raise 2.0pt \hbox{$\prod$}}\limits}
\def\sm{  \displaystyle\mathop{\raise 2.0pt \hbox{$\sum$}}\limits}

\let\sst=\scriptscriptstyle
\let\ds=\displaystyle
\def\lien{\mathrel{\mkern-4mu}}
\def\too{\relbar\lien\rightarrow}
\def\tooo{\relbar\lien\relbar\lien\too}
\def\toooo{\relbar\lien\relbar\lien\relbar\lien\too}

\def\order{\raise1.5pt \hbox{${\scriptscriptstyle \#}$}}

\begin{document}
 
\title[$p$-rationality of number fields]
{On $p$-rationality of number fields \\ Applications -- PARI/GP programs}

\author{Georges Gras}

\date{ To appear in Publ. Math. Fac. Sci. Besan\c con (2019)} 

\address{Villa la Gardette, Chemin Ch\^ateau Gagni\`ere 
 F--38520 Le Bourg d'Oisans, France 
{\rm \url{https://www.researchgate.net/profile/Georges_Gras}}}
\email{g.mn.gras@wanadoo.fr}

\begin{abstract} Let $K$ be a number field. We prove that its 
ray class group modulo $p^2$ (resp. $8$) if $p>2$ (resp. $p=2$) 
characterizes its $p$-rationality. Then we give two short,
very fast PARI Programs
(\S\S\,\ref{prog1}, \ref{prog2}) testing if $K$ (defined by an irreducible 
monic polynomial) is $p$-rational or not. 
For quadratic fields we verify some densities related to 
Cohen--Lenstra--Martinet ones and analyse Greenberg's 
conjecture on the existence of $p$-rational fields 
with Galois groups $(\Z/2\Z)^t$ needed for the construction of some 
Galois representations with open image. We give examples for 
$p=3$, $t=5$ and $t=6$ (\S\S\,\ref{real}, \ref{imaginary}) and illustrate 
other approaches (Pitoun--Varescon, Barbulescu--Ray). 
We conclude about the existence of imaginary quadratic fields, 
$p$-rational for all $p\geq 2$ (Angelakis--Stevenhagen on the 
concept of ``minimal absolute abelian Galois group'') which 
may enlighten a conjecture of $p$-rationality (Hajir--Maire) giving 
large Iwasawa $\mu$-invariants of some uniform pro-$p$-groups.

\medskip\noindent
{\sc R\'esum\'e} Soit $K$ un corps de nombres. Nous montrons que son
corps de classes de rayon modulo $p^2$ (resp. $8$) si $p>2$ (resp. $p=2$)
caract\'erise sa $p$-rationalit\'e. Puis nous donnons deux courts programmes
PARI (\S\S\,\ref{prog1}, \ref{prog2}) trs rapides  testant si $K$ 
(d\'efini par un polyn\^ome irr\'eductible unitaire) est $p$-rationnel ou non.
Pour les corps quadratiques nous v\'erifions certaines densit\'es en relation avec
celles de Cohen--Lenstra--Martinet et nous analysons la conjecture de 
Greenberg sur l'existence de corps $p$-rationnels de groupes de Galois 
$(\Z/2\Z)^t$ n\'ecessaires pour la construction de certaines repr\'esentations 
galoisiennes d'image ouverte. Nous donnons des exemples pour
$p=3$, $t=5$ et $t=6$ (\S\S\,\ref{real}, \ref{imaginary}) et illustrons d'autres
approches (Pitoun--Varescon, Barbulescu--Ray).
Nous concluons sur l'existence de corps quadratiques imaginaires 
$p$-rationnels pour tout $p\geq 2$ (Angelakis--Stevenhagen sur le concept
de  ``groupe de Galois ab\'elien absolu minimal'') qui peut \'eclairer une conjecture
de $p$-rationalit\'e (Hajir--Maire) donnant de grands invariants $\mu$ d'Iwasawa
relatifs \`a certains pro-$p$-groupes uniformes.
\end{abstract}

\maketitle

\section{Definition and properties of $p$-rationality}\label{def}

Let $K$ be a number field and let $p\geq 2$ be a fixed prime number. 
We denote by $\Cl_K$ the $p$-class group of $K$ in the ordinary sense
and by $E_K$ the group of $p$-principal global units $\varepsilon \equiv 1 
\pmod {\prod_{{\mathfrak p} \mid p} {\mathfrak p}}$ of $K$.

\smallskip
Let us describe the diagram of the so called {\it abelian $p$-ramification theory}
(from \cite[\S\,III.2 (c), \!Fig.~2.2]{Gr1}, \cite[Section 3]{Gr2}),
in which $\wt K$ is the compositum of the $\Z_p$-extensions of $K$, 
$H_K$ the $p$-Hilbert class field, $H_K^{\rm pr}$ the 
maximal abelian $p$-ramified (i.e., unramified outside~$p$)
pro-$p$-extension of $K$, then ${\rm Gal}(H_K^{\rm bp} / \wt K)$ is the 
Bertrandias--Payan module \cite[\S\,2, Diagramme 4]{Ja}.

\smallskip
See \cite{BJ} and \cite[Diagram 2]{Ja}  for 
a related context with logarithmic class groups.

\medskip
Let 
$$U_K:=\ds\plus_{{\mathfrak p} \div p} U_{\mathfrak p}^1$$
be the $\Z_p$-module of $p$-principal 
local units of $K$, where each 
$$U_{\mathfrak p}^1
:=\{u \in K_{\mathfrak p}^\times, \ u \equiv 1 \pmod {\ov {\mathfrak p}} \}$$ 
is the group of $\ov{\mathfrak p}$-principal units of the 
completion $K_{\mathfrak p}$ of $K$ at ${\mathfrak p} \mid p$,
where $\ov {\mathfrak p}$ is the maximal ideal of the ring of integers 
of $K_{\mathfrak p}$.
For any field $k$, let $\mu^{}_k$ be the $p$-group of roots of unity of $k$.

\smallskip
Then put 
$$\hbox{$W_K := {\rm tor}^{}_{\Z_p} \big (U_K  \big )  = \ds
\plus_{{\mathfrak p} \div p} \mu^{}_{K_{\mathfrak p}}\ $ and 
$\  {\mathcal W}_K := W_K /\mu^{}_K$. }$$
Let $\overline E_K$ be the closure in $U_K$ of the diagonal image 
of $E_K$; this gives in the diagram 
$${\rm Gal}(H_K^{\rm pr}/H_K) \simeq U_K/\overline E_K:$$
\unitlength=1.0cm 
$$\vbox{\hbox{\hspace{-2.8cm} 
 \begin{picture}(11.5,5.8)
\put(6.5,4.50){\line(1,0){1.3}}
\put(8.75,4.50){\line(1,0){2.0}}
\put(3.85,4.50){\line(1,0){1.4}}
\put(9.1,4.15){\footnotesize$\simeq\! {\mathcal W}_K$}
\put(4.25,2.50){\line(1,0){1.25}}

\bezier{350}(3.8,4.8)(7.6,5.8)(11.0,4.8)
\put(7.2,5.45){\footnotesize${\mathcal T}_K$}

\put(3.50,2.9){\line(0,1){1.25}}
\put(3.50,0.9){\line(0,1){1.25}}
\put(5.7,2.9){\line(0,1){1.25}}

\bezier{300}(3.9,0.5)(4.7,0.5)(5.6,2.3)
\put(5.2,1.3){\footnotesize$\simeq \! \Cl_K$}
\put(4.1,4.15){\footnotesize$\simeq\! \Cl_K^\infty$}

\bezier{300}(6.3,2.5)(8.5,2.6)(10.8,4.3)
\put(8.0,2.6){\footnotesize$\simeq \! U_K/\ov E_K$}

\bezier{500}(3.9,0.4)(9,0.4)(10.95,4.3)
\put(8.5,1.2){\footnotesize${\mathcal A}_K$}

\put(10.85,4.4){$H_K^{\rm pr}$}
\put(5.5,4.4){$\widetilde K\! H_K$}
\put(8.0,4.4){$H_K^{\rm bp}$}
\put(6.7,4.14){\footnotesize$\simeq\! {\mathcal R}_K$}
\put(3.3,4.4){$\widetilde K$}
\put(5.55,2.4){$H_K$}
\put(2.8,2.4){$\widetilde K \!\cap \! H_K$}
\put(3.4,0.40){$K$}
\end{picture}   }} $$
\unitlength=1.0cm

Put (\cite[Chapter III, \S\,(b) \& Theorem 2.5]{Gr1}):
$${\mathcal T}_K :=  {\rm  tor}^{}_{\Z_p}({\rm Gal}(H_K^{\rm pr}/K)); $$
in the viewpoint of Artin symbol, ${\mathcal T}_K $ is the kernel
of the ``logarithmic'' map:
$${\mathcal A}_K \mathop {\toooo}^{{\rm Art}^{-1}} {\mathcal I}_K/ {\mathcal P}_{K,\infty} 
 \mathop {\toooo}^{{\rm Log}} \Z_p {\rm Log}(I_K) \subseteq 
\Big( \plus_{{\mathfrak p} \div p} K_{\mathfrak p}\Big) \Big / \Q_p\, {\rm log}(E_K),$$

where ${\mathcal I}_K := I_K \otimes \Z_p$, $I_K$ being the group 
of prime to $p$ ideals of $K$, and ${\mathcal P}_{K,\infty}$ the group of infinitesimal 
principal ideals; ${\rm log}$ is the $p$-adic logarithm, and 
${\rm Log}({\mathfrak a}) := \frac{1}{m}\, {\rm log}(\alpha) \!\!\pmod {\Q_p\, {\rm log}(E_K)}$ 
for any relation  in $I_K$ of the form:
$$\hbox{${\mathfrak a}^m=(\alpha)$, $\ m\in \Z,\ \alpha \in K^\times$.} $$
Let $\Cl_K^\infty$ be the subgroup of $\Cl_K$ corresponding, by class field theory, 
to ${\rm Gal}(H_K/\wt K \cap H_K)$.

\smallskip
We have, under the Leopoldt conjecture, 
the exact sequence defining ${\mathcal R}_K$
(\cite[Lemma III.4.2.4]{Gr1} or \cite[Lemma 3.1 \& \S\,5]{Gr2}):
\begin{equation*}
1\to  {\mathcal W}_K  \too 
 {\rm tor}_{\Z_p}^{} \big(U_K \big / \ov E_K \big) \\
 \mathop {\tooo}^{{\rm log}}  {\rm tor}_{\Z_p}^{}\big({\rm log}\big 
(U_K \big) \big / {\rm log} (\ov E_K) \big)=: {\mathcal R}_K \to 0. 
\end{equation*}

The group ${\mathcal R}_K$ is called the {\it normalized $p$-adic 
regulator of $K$} and makes sense for any number field.

\begin{definition} The field $K$ is said to be $p$-rational if 
the Leopoldt conjecture is satisfied for $p$ in $K$ and if
the torsion group ${\mathcal T}_K$ is trivial.
\end{definition}

\begin{theorem} \label{cns} 
(\cite[Th\'eor\`eme et D\'efinition 4.1]{GJ} or \cite[Theorem IV.3.5]{Gr1}).
For  a number field $K$, each of the following properties
is equivalent to its $p$-rationality (where $2\,r_2$ is the 
number of complex embeddings of $K$):

\smallskip
(i) ${\mathcal A}_K := {\rm Gal}(H_K^{\rm pr}/K) \simeq \Z_p^{r_2+1}$,  

\smallskip
(ii) the Galois group ${\mathcal G}_K$ of the maximal $p$-ramified  
pro-$p$-extension of $K$ is a free pro-$p$-group on 
$r^{}_2+1$ generators (i.e., ${\rm H}^2({\mathcal G}_K, \Z/p\,\Z)=1$),

\smallskip
(iii) we have the following alternative:

\smallskip
\quad $\ \bullet\ $ either $\mu_p^{} \subset K$, the set of $p$-places of $K$
is a singleton $\{{\mathfrak p}\}$, and ${\mathfrak p}$ generates the 
$p$-class group of $K$ (in the restricted sense for $p=2$),

\smallskip
\quad $\ \bullet\ $ or $\mu_p^{} \not \subset K$ (whence $p \ne 2$), no prime ideal 
${\mathfrak p} \mid p$ of $K$ is totally split in $K(\mu_p^{})/K$
and the $\omega$-components of the $p$-classes
of the ${\mathfrak P} \mid p$ in $K(\mu_p^{})$ generate the 
$\omega$-component of the $p$-class group of $K(\mu_p^{})$, where $\omega$
is the Teichm\"uller character.
\end{theorem}

We can give, for $p=2$ and $p=3$, more elaborate statements as follows:

\begin{example} \label{p=2}{\rm
From \cite[\S\,III.2, Corollary to Theorem 2]{Gr4} using properties
of the ``regular kernel'' of the ${\rm K}_2(K)$ of a number field,
or \cite[Example IV.3.5.1]{Gr1} from properties of the groups ${\mathcal T}_K$,
we can characterize the $2$-rationality of $2$-extensions of $\Q$, 
independently of (iii):

\smallskip
{\it The set of abelian $2$-rational $2$-extensions of $\Q$ are all the subfields
of the compositum $\Q(\mu_{2^\infty}^{}) \cdot \Q(\sqrt{-\ell})$, $\ell \equiv 3
\pmod 8$ prime, and all the subfields of the compositum $\Q(\mu_{2^\infty}^{}) \cdot 
\Q \Big(\sqrt{\sqrt{\ell} \,(a - \sqrt{\ell}) /2} \,\Big)$, $\ell = a^2+b^2 \equiv 5 \pmod 8$ 
prime, $a$ odd. }

\smallskip
For quadratic fields this gives, for any 
$\ell \equiv \pm 3 \pmod 8$:

\smallskip
\centerline{$K \in \{\Q(\sqrt {-1}), \ \Q(\sqrt {2}),\  \Q(\sqrt {-2})$,\  $\Q(\sqrt {\ell}), \ 
\Q(\sqrt {-\ell}),\  \Q(\sqrt {2 \ell}),\  \Q(\sqrt {-2 \ell})\}$.}}
\end{example}

\begin{example} \label{p=3}{\rm
In the same way, the set of abelian $3$-rational $3$-extensions 
of $\Q$ are all the subfields of the compositum of 
any cubic field of prime conductor $\ell  =\frac{a^2+27\,b^2}{4} 
\equiv 4 \ {\rm or}\  7\! \pmod 9$ (for which a defining monic polynomial is 
$x^3+x^2-\frac{\ell-1}{3}\,x- \frac{\ell\,(a+3)-1}{27}$, $a \equiv 1\! \pmod 3$),
with the cyclotomic $\Z_3$-extension.}
\end{example}

\begin{example} \label{spiegel}
{\rm Consider the prime $p=3$ and a quadratic field $K=\Q(\sqrt d)$, 
$d \in \Z$, $K\ne \Q(\sqrt{-3})$; 
in this case the $\omega$-component of the $3$-class group of 
$K(\mu_3^{}) = K(\sqrt {-3})$ is isomorphic to the $3$-class group of 
the mirror field $K^* := \Q(\sqrt{-3 \, d})$. Moreover the 
$3$-class of ${\mathfrak p} \mid 3$ in $\Q(\sqrt{-3 \, d})$ is trivial since 
$3$ is ramified or inert in this extension under the non-splitting 
assumption in $K(\mu_3^{})/K$. Thus in that case, the $3$-rationality of 
$K$ is equivalent to fulfill the following two conditions for $K\ne \Q(\sqrt{-3})$:

\smallskip
(i) The $3$-class group of $K^* = \Q(\sqrt{-3 \, d})$ is trivial,

\smallskip
(ii) we have $d \equiv \pm 1 \pmod 3$ (i.e.,  ramification of $3$ in $K^*$) 
or $d \equiv 3 \!\pmod 9$ (i.e., inertia of $3$ in $K^*$).

\smallskip
If $d \equiv 6 \pmod 9$, the prime
ideal above $3$ in $K$ splits in $K(\sqrt{-3})$ so that the non-$3$-rationality
of $K$ comes from the factor 
$\bigoplus_{{\mathfrak p} \div p} \mu^{}_{K_{\mathfrak p}} \big / \mu_K^{}
= \mu^{}_{K_{\mathfrak p}} \simeq  \Z/ 3\,\Z$ for the unique 
${\mathfrak p} \mid 3$ in $K$.}
\end{example}

\begin{remarks} \label{remas}
(a) Under Leopoldt's conjecture, the transfer homomorphisms
${\mathcal T}_k \to {\mathcal T}_K$ are injective in any extension $K/k$
of number fields \cite[Theorem IV.2.1]{Gr1}; so if $K$ is $p$-rational, 
all the subfields of $K$ are $p$-rational. If $p \nmid [K : \Q]$, the
norms ${\rm N}_{K/k} : {\mathcal T}_K \to {\mathcal T}_k$ 
(which correspond to the restrictions
${\mathcal T}_K \to {\mathcal T}_k$ by class field theory) are 
surjective for all $k \subseteq K$.

\smallskip
When $K/\Q$ is Galois and $p \nmid [K : \Q]$, the reciprocal
about the $p$-rationalities of subfields of $K$ is true
for some Galois groups $G:= {\rm Gal}(K/\Q)$ and some families of subfields.
This occurs for instance when $K/\Q$ is abelian with the family of 
all maximal cyclic subfields of $K$: indeed, as $p \nmid [K : \Q]$,
${\mathcal T}_K \simeq \bigoplus_\chi {\mathcal T}_K^{e_\chi}$,
where $\chi$ runs trough the set of
rational characters of $G$, $e_\chi$ being the 
corresponding idempotent; then ${\mathcal T}_K^{e_\chi}$
is isomorphic to a submodule of ${\mathcal T}_{k_\chi}$
where $k_\chi$ (cyclic) is the subfield of $K$ 
fixed by the kernel of $\chi$.

\smallskip
For a compositum $K$ of quadratic fields, this
means that, for $p>2$, $K$ is $p$-rational if and only if all the
quadratic subfields of $K$ are $p$-rational.

\smallskip
\quad (b) When $K$ is a $p$-extension of a $p$-rational field $k$, 
$K$ is $p$-ratio\-nal if and only if the extension $K/k$ is 
{\it $p$-primitively ramified}. This notion, defined first in 
\cite[Chapter III, \S\,1 \& \S\,2]{Gr4} form our Crelle's papers on 
$p$-ramification and in \cite[Section 1]{GJ}, 
contains the case of $p$-ramification 
and some other explicit cases as in Examples \ref{p=2} and \ref{p=3}; 
this notion has been extensively applied in 
\cite[Theorem IV.3.3, Definition IV.3.4]{Gr1}, \cite{Ja, JN, MN, Mo}.

\smallskip
\quad (c) In \cite{Gre}, abelian $\ell$-extensions ($\ell \ne p$ prime) are used
for applications to continuous Galois representations
${\rm Gal}(\ov \Q/\Q) \to {\rm GL}_n(\Z_p)$
with open image for which one needs the condition (ii) of 
Theorem \ref{cns} (i.e., the $p$-rationality), and where some
properties given in the above references are proved again.

\smallskip
\quad (d) We have conjectured in \cite{Gr3} that $K$ is $p$-rational 
for all $p \gg 0$, but in practical applications one works with small primes; 
so, care has been taken to consider also the case $p=2$ in the programs.
We ignore what kind of heuristics (in the meaning of many works such as
\cite{CL, CM, DJ, FK, PS} on class groups, Tate-Shafarevich groups,\,$\ldots$) 
are relevant for the ${\mathcal T}_K$ when $p$ varies. 
This should be very interesting since the ${\mathcal T}_K$ mix classes and units.
\end{remarks}

From the general schema, ${\mathcal T}_K=1$ if and only if the 
three invariants ${\mathcal W}_K$, ${\mathcal R}_K$ 
and $\Cl_K^\infty$, are trivial. 
Thus, in some cases, it may be possible to test each of these trivialities,
depending on the knowledge of the field $K$; for instance, assuming the
$p$-class group trivial and $p$ unramified, the computation of ${\mathcal W}_K$ 
is purely local and that of the normalized $p$-adic regulator ${\mathcal R}_K$,
closely related to the classical $p$-adic regulator,
may be given in a specific program since it is the most unpredictable invariant. 
But a field given by means of a polynomial $P$ may be more mysterious
regarding these three factors.

\section{General theoretical test of $p$-rationality}

\subsection{Test using a suitable ray class group}
In \cite{Pi} and \cite[Theorem 3.11 \& Corollary 4.1]{PV}, 
are given analogous methods for
the general computation of the structure of ${\mathcal T}_K$, 
but here we need only to characterize the triviality (or not) of 
${\mathcal T}_K$ in the relation:

\medskip
\centerline{${\mathcal A}_K := {\rm Gal}(H_K^{\rm pr}/K) \simeq 
\Z_p^r \times {\mathcal T}_K$,}

\medskip
where $r=r_2+1$, $2\,r_2$ being the number of complex embeddings of $K$. 
As ${\mathcal T}_K$ is a direct factor in ${\mathcal A}_K$, the
structure of the whole Galois group ${\mathcal A}_K$ may be analized at 
a finite step by means of the Galois group of
a suitable ray class field $K(p^n)$ of modulus $(p^n)$. 
Since PARI gives the structure of ray class groups 
$\Cl_K(p^n):={\rm Gal}(K(p^n)/K)$,
the test of $p$-rationality is obtained for $n$ large enough as follows:
if $\Cl_K(p^n)$ has a $p$-rank such that:
$${\rm rk}_p(\Cl_K(p^n)) \geq r+1,$$
then  $K$ is not $p$-rational since ${\rm Gal}(\wt K/K)$ has $p$-rank $r$.
The minimal $n_0$ needed for the test is given as a consequence of the 
following result:

\begin{theorem} \label{thmfond}
For any ${\mathfrak p} \mid p$ in $K$ and any 
$j\geq 1$, let $U_{\mathfrak p}^j$
be the group of local units $1+ \ov{\mathfrak p}^j$, where 
$\ov{\mathfrak p}$ is the maximal ideal of the ring of integers of $K_{\mathfrak p}$. 

\smallskip
For a modulus of the form $(p^n)$, $n\geq 0$, let
$\Cl_K(p^n)$ be the corresponding ray class group. Then
for $m \geq n \geq 0$, we have the inequalities:
$$0 \leq {\rm rk}_p(\Cl_K(p^m)) - {\rm rk}_p(\Cl_K(p^n))
\leq \sm_{{\mathfrak p} \mid p} {\rm rk}_p 
\big((U_{\mathfrak p}^1)^p\,U_{\mathfrak p}^{n \cdot e_{\mathfrak p}}/
(U_{\mathfrak p}^1)^p\,U_{\mathfrak p}^{m \cdot e_{\mathfrak p}} \big), $$
where $e_{\mathfrak p}$ is the ramification index of ${\mathfrak p}$ in $K/\Q$.
\end{theorem}

\begin{proof} From \cite[Theorem I.4.5 \& Corollary I.4.5.4]{Gr1} taking for
$T$ the set of $p$-places and for $S$ the set of real infinite places 
(ordinary sense).
\end{proof}

\begin{corollary} \label{corofond}
We have ${\rm rk}_p(\Cl_K(p^m)) = 
{\rm rk}_p(\Cl_K(p^n)) = {\rm rk}_p({\mathcal A}_K)$ for all 
$m \geq n \geq n_0$, where $n_0=3$ for $p=2$ and $n_0=2$ for $p >2$. 

\smallskip
Thus $K$ is $p$-rational if and only if 
${\rm rk}_p(\Cl_K(p^{n_0}))=r$, where $r=r_2+1$.
\end{corollary}

\begin{proof} It is sufficient to get, for some fixed $n \geq 0$:
$$\hbox{$(U_{\mathfrak p}^1)^p\,U_{\mathfrak p}^{n \cdot e_{\mathfrak p}} 
= (U_{\mathfrak p}^1)^p$, \  for all ${\mathfrak p} \mid p$, }$$
hence $U_{\mathfrak p}^{n \cdot e_{\mathfrak p}} \subseteq 
(U_{\mathfrak p}^1)^p$ for all ${\mathfrak p} \mid p$; 
indeed, we then have:
$$\hbox{${\rm rk}_p(\Cl_K(p^n)) = {\rm rk}_p(\Cl_K(p^m)) = 
r+ {\rm rk}_p({\mathcal T}_K)$ as $m\to\infty$, } $$
giving ${\rm rk}_p(\Cl_K(p^n)) = r+ {\rm rk}_p({\mathcal T}_K)$ for such $n$. 

\medskip
The condition $U_{\mathfrak p}^{n \cdot e_{\mathfrak p}} \subseteq 
(U_{\mathfrak p}^1)^p$ is fulfilled as soon as
$n \cdot e_{\mathfrak p} > \Frac{p\cdot e_{\mathfrak p}}{p-1}$, hence:
$$n > \Frac{p}{p-1}$$
(see \cite[Chap. I, \S\,5.8, Corollary 2]{FV}
or \cite[Proposition 5.7]{W}; a fact also used in \cite[Proposition 1.13]{HM}), 
whence the value of $n_0$; furthermore, $\Cl_K(p^{n_0})$ gives the 
exact $p$-rank of ${\mathcal T}_K$.
\end{proof}

\subsection{Basic invariants of $K$ with PARI \cite{P}}
The reader whishing to test only $p$-rationalities may 
go directly to Subsections \ref{prog1}, \ref{prog2}.

\smallskip
The examples shall be given with the following polynomial defining $K$:
$$P=x^3-5\,x+3$$
(recall that for PARI, $P$ must be monic and in $\Z[x]$),
for which one gives the classical information 
(test of irreducibility, Galois group of the Galois closure of $K$, 
discriminant of $K$) which are the following,
with the PARI responses:

\smallskip\footnotesize
\begin{verbatim}
polisirreducible(x^3-5*x+3)
   1
polgalois(x^3-5*x+3)
   [6,-1,1,"S3"]
factor(nfdisc(x^3-5*x+3))
   [257 1]
\end{verbatim}

\normalsize\smallskip\noindent
showing that $P$ is irreducible, that the Galois closure of $K$ is
the diedral group of order $6$, and that
the discriminant of $K$ is the prime $257$.

\smallskip
Then $K$ is precised by the signature $[r_1, r_2]$, 
the structure of the whole class group and the fundamental units.

\smallskip\footnotesize
\begin{verbatim}
{P=x^3-5*x+3;p=2;K=bnfinit(P,1);C7=component(K,7);C8=component(K,8);
Sign=component(C7,2);print("Signature of K: ",Sign);
print("Structure of the class group: ",component(C8,1));
print("Fundamental system of units: ",component(C8,5))}
[3, 0],  [1, [], []],  [x-1, x^2+2*x-2]
\end{verbatim}

\normalsize\smallskip\noindent
giving a totally real field, a trivial class group and the two fundamental units.
From $K=bnfinit(P,1)$ and $C8=component(K,8)$, the regulator of $K$ is given 
by $component(C8,1)$.
The next program gives, for information, the decomposition of some primes $p$:

\smallskip\footnotesize
\begin{verbatim}
{P=x^3-5*x+3;K=bnfinit(P,1);
forprime(p=2,13,print(p," ",idealfactor(K,p)))}
2  Mat([[2,[2,0,0]~,1,3,1],1])
3  [[3,[0,0,1]~,1,1,[1,1,-1]~],1;[3,[4,1,-1]~,1,2,[0,0,1]~],1]
5  [[5,[2,0,1]~,1,1,[2,1,2]~],1;[5,[2,1,-3]~,1,2,[2,0,1]~],1]
7  [[7,[1,0,1]~,1,1,[-1,1,-2]~],1;[7,[6,1,-2]~,1,2,[1,0,1]~],1]
11 Mat([[11,[11,0,0]~,1,3,1],1])
13 Mat([[13,[13,0,0]~,1,3,1],1])
\end{verbatim}

\normalsize\smallskip\noindent
showing that $2, 11, 13$ are inert, that $3, 5, 7$ split into two 
prime ideals with residue degrees 1 and 2, respectively. 
Taking $p=257$, one obtains:

\smallskip\footnotesize
\begin{verbatim}
257 [[257,[-101,0,1]~,1,1,[-81,1,100]~],1;
                                     [257,[-78,0,1]~,2,1,[-86,1,77]~],2]
\end{verbatim}

\normalsize\smallskip\noindent
which splits into ${\mathfrak p}_1 \cdot {\mathfrak p}_2^2$.

\smallskip
To obtain a polynomial from a compositum of known fields 
(e.g., quadratic fields) one uses by induction the instruction
$polcompositum$:

\smallskip\footnotesize
\begin{verbatim}
{P1=x^2-2;P2=x^2+5;P3=x^2-7;P=polcompositum(P1,P2);P=component(P,1);
P=polcompositum(P,P3);P=component(P,1);print(P)}
x^8-16*x^6+344*x^4+2240*x^2+19600
\end{verbatim}

\normalsize\smallskip
The instruction giving the structure of $\Cl_K(p^{n})$ 
is the following (we compute the structure of the ray 
class groups with modulus $p^n$, up to $n=5$, to see the 
stabilization of the $p$-ranks; this would give 
the group invariants of ${\mathcal T}_K$, 
as is done in \cite{Pi, PV}):

\smallskip\footnotesize
\begin{verbatim}
{K=bnfinit(x^3-5*x+3,1);p=2;for(n=0,5,Hpn=bnrinit(K,p^n);
print(n," ",component(Hpn,5)))}
p=2                 p=3               p=257
0  [1,[]]           0  [1,[]]         0  [1,[]]
1  [1,[]]           1  [1,[]]         1  [128,[128]]
2  [2,[2]]          2  [3,[3]]        2  [32896,[32896]]
3  [4,[2,2]]        3  [9,[9]]        3  [8454272,[8454272]]
4  [8,[4,2]]        4  [27,[27]]      4  [2172747904,[2172747904]]
5  [16,[8,2]]       5  [81,[81]]      5  [558396211328,[558396211328]]
\end{verbatim}

\normalsize\smallskip\noindent
(where $32896=2^7 \cdot 257$, $8454272=2^7 \cdot 257^2$, 
$2172747904 =2^7 \cdot 257^3$,\,$\ldots$):

\smallskip
Thus $K$ is $p$-rational for $p=3$ and $257$, but not for $p=2$.

\smallskip
Now we give the case of the first irregular prime $p=37$
for which we know that the $p$th cyclotomic field is
not $p$-rational (indeed, ${\mathcal T}_K=1$ is equivalent
to $\Cl_K=1$ for the $p$th cyclotomic fields \cite[Th\'eor\`eme \& D\'efinition 2.1]{GJ}). 
So we must see that the $p$-ranks are equal to $19+1$, at least for $n\geq 2$:

\smallskip
\footnotesize
\begin{verbatim}
{p=37;K=bnfinit(polcyclo(p),1);
for(n=0,2,Hpn=bnrinit(K,p^n);print(n," ",component(Hpn,5)))}
0 [37,[37]]
1 [624931990990842127748277129373,[1369,37,37,37,37,37,37,37,37,37,
   37,37,37,37,37,37,37,37]]
2 [390539993363777986320898213181845819006713655084697379373129,
  [50653,1369,1369,1369,1369,1369,1369,1369,1369,1369,1369,1369,
  1369,1369,1369,1369,1369,37,37,37]]
\end{verbatim}

\normalsize\smallskip\noindent
where $1369=37^2, 50653=37^3$.

\section{Full general PARI programs testing $p$-rationality}\label{prog}
We bring together some instructions given in the previous section and recall
that, in the programs, $n=2$ (resp. $3$) if $p\ne 2$ (resp. $p=2$).

\subsection{\!General program with main invariants and test of $p$-rationality}\label{prog1}
The reader has to introduce an {\it irreducible monic polynomial} $P \in \Z[x]$ and a 
{\it prime number} $p \geq 2$. For $p=2$ the $p$-rationality is in the 
ordinary sense.

\smallskip
\footnotesize
\begin{verbatim}
====================================================================
  {P=x^3-5*x+3;p=2;K=bnfinit(P,1);
  Sign=component(component(K,7),2);print("Signature of K: ",Sign);
  print("Galois group of the Galois closure of K: ",polgalois(P));
  print("Discriminant: ",factor(component (component(K,7), 3)));
  print("Structure of the class group: ",component(component(K,8),1));
  print("Fundamental system of units: ",component(component(K,8),5));
  r=component(Sign,2)+1;n=2;if(p==2,n=3);
  print(p,"-rank of the compositum of the Z_",p,"-extensions: ",r);
  Hpn=component(component(bnrinit(K,p^n),5),2);L=listcreate;
  e=component(matsize(Hpn),2);R=0;for(k=1,e,c=component(Hpn,e-k+1);
  if(Mod(c,p)==0,R=R+1;listinsert(L,p^valuation(c,p),1)));
  print("Structure of the ray class group mod ",p,"^n: ",L);
  if(R>r,print("rk(T)=",R-r," K is not ",p,"-rational"));
  if(R==r,print("rk(T)=",R-r," K is ",p,"-rational"))}
====================================================================
\end{verbatim}

\normalsize\smallskip
(i) With the above data, $K$ is not $2$-rational nor $293$-rational
(up to $p \leq 10^5$).

\smallskip
(ii) The field $K=\Q(\sqrt{-161})$ is not $2$-rational 
(one may prove that the $2$-Hilbert class field is linearly disjoint from $\wt K$ 
\cite[Example III.6.7]{Gr1}).

\smallskip
(iii) The field $K=\Q(\sqrt{69})$ is not $3$-rational (comes
from ${\mathcal R}_K\ne1$).

\smallskip
(iv) The quartic cyclic fields $K$ defined by $P=x^4+5\,x^2+5$
(conductor $5$) and by $P=x^4+13\,x^2+13$ (conductor $13$) 
are $2$-rational (see Example \ref{p=2}).

\subsection{Test of $p$-rationality (simplified programs)}\label{prog2}
Since some parts of the first program are useless and some computations 
intricate (e.g., Galois groups in large degrees), one may use the following 
simplified program to test only the $p$-rationality. 

\smallskip
We test the $3$-rationality of 
$K=\Q(\sqrt{-1}, \sqrt{2}$, $\sqrt{5}, \sqrt{11}, \sqrt{97})$
given in \cite{Gre} ($P$ is computed from the instruction 
$polcom\-positum$; the program takes $19\, min,\, 26.663\, ms$ and need 
$allocatemem(800000000)$):

\smallskip
\begin{verbatim}
Programme I (define the polynomial P and the prime p):
\end{verbatim}

\smallskip\footnotesize
\begin{verbatim}
====================================================================
{P=x^32-1824*x^30+1504544*x^28-743642240*x^26+246039201472*x^24
-57656224594432*x^22+9874427075761664*x^20-1257037661975865344*x^18
+119781806181353182720*x^16-8534335878932228562944*x^14
+450658848166023111041024*x^12-17330171952567833219399680*x^10
+471547188605910876106571776*x^8-8678484783929508254539710464*x^6
+100678473375628844348283158528*x^4-658128522558747992210233884672*x^2
+1995918433518957384065860304896;K=bnfinit(P,1);p=3;n=2;if(p==2,n=3);
Kpn=bnrinit(K,p^n);S=component(component(Kpn,1),7);
r=component(component(S,2),2)+1;
print(p,"-rank of the compositum of the Z_",p,"-extensions: ",r);
Hpn=component(component(Kpn,5),2);L=listcreate;
e=component(matsize(Hpn),2);R=0;for(k=1,e,c=component(Hpn,e-k+1);
if(Mod(c,p)==0,R=R+1;listinsert(L,p^valuation(c,p),1)));
print("Structure of the ",p,"-ray class group: ",L);
if(R>r,print("rk(T)=",R-r," K is not ",p,"-rational"));
if(R==r,print("rk(T)=",R-r," K is ",p,"-rational"))}
====================================================================
3-rank of the compositum of the Z_3-extensions: 17
Structure of the 3-ray class group:List([3,3,3,3,3,3,3,3,3,3,3,3,3,3,9,9,9])
K is 3-rational
\end{verbatim}

\normalsize\smallskip
If one whises to test the $p$-rationality of $K$ for $p$ varying in an 
interval $[b, B]$, it is necessary to compute first the data $bnfinit(P,1)$ 
which is independent of $p$ and takes lots of time. Then the tests for 
$p$-rationalities are very fast. 

\smallskip
This gives the following writing where we give the $p$-structure of
$\Cl_K(p^{n_0})$ up to $p \leq 100$, only for the non-$p$-rational cases:

\smallskip
\begin{verbatim}
Programme II (define P and the interval [b,B] of primes p):
\end{verbatim}

\smallskip\footnotesize
\begin{verbatim}
====================================================================
{P=x^32-1824*x^30+1504544*x^28-743642240*x^26+246039201472*x^24
-57656224594432*x^22+9874427075761664*x^20-1257037661975865344*x^18
+119781806181353182720*x^16-8534335878932228562944*x^14
+450658848166023111041024*x^12-17330171952567833219399680*x^10
+471547188605910876106571776*x^8-8678484783929508254539710464*x^6
+100678473375628844348283158528*x^4-658128522558747992210233884672*x^2
+1995918433518957384065860304896;K=bnfinit(P,1);b=2;B=100;
r=component(component(component(K,7),2),2)+1;
print("p-rank of the compositum of the Z_p-extensions: ",r);
forprime(p=b,B,n=2;if(p==2,n=3); 
Kpn=bnrinit(K,p^n);Hpn=component(component(Kpn,5),2);
L=listcreate;e=component(matsize(Hpn),2);
R=0;for(k=1,e,c=component(Hpn,e-k+1);if(Mod(c,p)==0,R=R+1; 
listinsert(L,p^valuation(c,p),1)));
print("Structure of the ",p,"-ray class group: ",L); 
if(R>r,print("rk(T)=",R-r," K is not ",p,"-rational"));
if(R==r,print("rk(T)=",R-r," K is ",p,"-rational")))}
====================================================================
p-rank of the compositum of the Z_p-extensions:17

List([2,2,2,2,2,4,4,4,4,4,4,4,8,8,8,8,8,8,8,16,16,16,32,32,64,128])
rk(T)=9 K not 2-rational
List([7,7,7,7,7,7,7,7,7,7,7,7,7,7,7,7,7,7])                
rk(T)=1 K not 7-rational
List([11,11,11,11,11,11,11,11,11,11,11,121,121,121,121,121,121,121,1331])  
rk(T)=2 K not 11-rational
List([13,13,13,13,13,13,13,13,13,13,13,13,13,13,13,13,13,13])      
rk(T)=1 K not 13-rational
List([17,17,17,17,17,17,17,17,17,17,17,17,17,17,17,17,17,17])      
rk(T)=1 K not 17-rational
List([19,19,19,19,19,19,19,19,19,19,19,19,19,19,19,19,19,19])      
rk(T)=1 K not 19-rational
List([29,29,29,29,29,29,29,29,29,29,29,29,29,29,29,29,29,29])       
rk(T)=1 K not  29-rational
List([31,31,31,31,31,31,31,31,31,31,31,31,31,31,31,31,31,31])      
rk(T)=1 K not 31-rational
List([43,43,43,43,43,43,43,43,43,43,43,43,43,43,43,43,43,43])      
rk(T)=1 K not 43-rational
List([73,73,73,73,73,73,73,73,73,73,73,73,73,73,73,73,73,73])      
rk(T)=1 K not 73-rational
\end{verbatim}

\normalsize\smallskip
The $p$-rationality holds for $3,5,23,37,41,47,53,59,61,67,71,79,83,89,97$
and we find that $K$ is not $p$-rational up to $10^5$ for the primes:

\smallskip
$p \in \{2, 7,11,13,17,19, 29, 31,43, 73, 163, 191, 263, 
409, 571, 643, 1049$, 

\qquad $2671, 3331, 3917, 6673, 8941, 28477, 36899, 39139, 85601, 99149$,

\qquad  $134339, 203393, 231901, 283979, 353711, 363719\}$.

\smallskip
The above cases of non-$p$-rationality are numerous, because of the 
important number of maximal cyclic subfields of $K$, and are essentially 
due to some units; for instance, for the unit $\varepsilon = 1+\sqrt 2 \in E_K$ 
and $p=31$, we get:
$$\varepsilon^{30}=152139002499+107578520350 \,\sqrt 2
\equiv 1 \pmod {31^2}, $$
which means ${\mathcal R}_K =
\frac{1}{31}\, {\rm log}(\varepsilon) \equiv 0 \pmod {31}$.

\medskip
We must note that, once the important instruction $K=bnfinit(P,1)$ is 
done by PARI, the execution time for various $p$ is much shorter.

\section{Heuristics on Greenberg's conjecture (with $p=3$)}

\subsection{Generalities}
A conjecture, in relation with the construction of continuous Galois 
representations ${\rm Gal}(\ov \Q/\Q) \too {\rm GL}_n(\Z_p)$ with 
open image, is the following, stated in the case of a compositum 
of quadratic fields \cite[Conjecture 4.2.1]{Gre}:

\begin{conjecture} \label{GC}
For any odd prime $p$ and for any $t \geq 1$ there
exists a $p$-rational field $K$ such that ${\rm Gal}(K/\Q) \simeq (\Z/2\Z)^t$.
\end{conjecture}

For the link between $n$ and $t$ and more information, see
\cite[Propositions 6.1.1 \& 6.2.2]{Gre}.
For simplicity, we shall discuss this conjecture for $p=3$
(for other cases, see \cite{BR}); 
then the $3$-rationality of a compositum $K$ of $t$
quadratic fields is equivalent to the following condition
(from Example \ref{spiegel}):

\smallskip
{\it For all the quadratic subfields $k =: \Q(\sqrt{d}) \ne \Q(\sqrt{-3})$ 
of $K$, the $3$-class group of the mirror field $k^* := \Q(\sqrt {-3\cdot d})$ 
is trivial and $3$ does not split in $k^*/\Q$.}

\smallskip
This needs $2^t-1$ conditions of $3$-rationality.

\subsection{Technical conditions} The case $t=1$ is clear
and gives infinitely many $3$-rational fields if we refer to classical
heuristics \cite{CL, CM}.
When $t\geq 3$ we must assume that $\Q(\sqrt{-3})$
is not contained in $K$, otherwise, from Theorem \ref{cns}\,(iii),
if $K_0$ is the inertia field of $3$ in $K/\Q$, then $[K : K_0]=2$,
$[K_0 : \Q]\geq 4$,
and necessarily $3$ splits in part in $K_0/\Q$ (the factor 
${\mathcal W}_K$ is non-trivial). 

\smallskip
For the biquadratic field $\Q(\sqrt d, \sqrt{-3})$, $d \not\equiv 0 \!\!\pmod 3$,
the $3$-rationality holds if and only if $d\equiv -1 \pmod 3$ and the 
$3$-class group of the imaginary field $\Q(\sqrt{d})$ or $\Q(\sqrt{-3\,d})$
is trivial (using Scholz's theorem). 
In the sequel, we shall assume that $K$ does not contain $\Q(\sqrt{-3})$
and that $t \geq 3$.

\smallskip
Let $K=\Q(\sqrt{d_1}, \ldots , \sqrt{d_t})$ be such that 
${\rm Gal}(K/\Q) \simeq (\Z/2\Z)^t$, $t\geq 3$. We denote by $k=\Q(\sqrt d)$
any quadratic subfield of $K$ and by $k^*$ the ``mirror field''
$\Q(\sqrt{-3 \cdot d})$. We then have the obvious lemma:

\begin{lemma} The conditions of non-splitting of $3$ in all
the extensions $k^*/\Q$, $k = \Q(\sqrt d) \subseteq K$, are satisfied 
if and only if for all the integers $d$ we have $3 \nmid d$ or 
$d \equiv 3 \pmod 9$.
\end{lemma}

\begin{corollary}\label{congruences}
(i) If $3$ is unramified in $K/\Q$,
then $K=\Q(\sqrt{d_1^{}}, \ldots , \sqrt{d_t^{}})$ with
$d_i \equiv \pm 1 \pmod 3$, for $i=1, \ldots , t$.

\smallskip
(ii) If $3$ is ramified in $K/\Q$ then $K$ is a direct compositum 
of the form $K_0\, \Q(\sqrt{3\,d_t})$, 
where $3$ is unramified in $K_0$ and $d_t \not \equiv 0 \pmod 3$,
whence $K=\Q(\sqrt{d_1^{}}, \ldots , \sqrt{d_{t-1}}, \sqrt{3 \, d_t^{}})$
with $d_i \equiv 1 \pmod 3$ for $i=1, \ldots, t$. 
\end{corollary}

Case (ii) comes from the fact that for any 
$d \in \langle d_1, \ldots , d_{t-1} \rangle\cdot \Q^{\times 2}$,
$d \not \equiv 0 \pmod 3$,
the consideration of the subfield $k=\Q(\sqrt{3 \,d_t \,d})$ needs, for 
$k^*=\Q(\sqrt{- d_t \,d})$, the condition $d_t \, d \equiv 1 \pmod 3$,
whence the hypothesis as soon as $t \geq 3$.

\medskip
So whatever the form of $K$, the results depend on heuristics on the $3$-class
groups of the $2^t-1$ fields $k^*$ when random integers $d_i$ are
given with the above conditions.

\smallskip
A first heuristic is to assume that a $3$-class group does not depend on 
the above assumption of decomposition of the prime $3$; this is natural
since only genera theory (i.e., when $p=2$) is concerned with such problem.
Of course, all the $2^t-1$ integers $d$ are not random, but the classical
studies of repartition of class groups show that the two phenomena
(an integer $d$ and a $3$-class group) are independent.

\subsection{Densities of $3$-rational quadratic fields}${}$
We give now densities of $3$-rational quadratic fields $K=\Q(\sqrt d)$. 
There are two algorithms for this: to use the general Program I of 
\S\,\ref{prog2} or to use the characterization given by Example \ref{spiegel}. 
It is easy to see that Program I takes the same time 
for real or imaginary fields and that the characterization using
the computation of class numbers of the mirror fields $K^*$
is faster for real $K$. 

\smallskip
We shall compare with the heuristics of
Cohen--Lenstra--Martinet (see \cite{CL, CM}).

\smallskip
So we obtain the following programs in which
$N$ (resp. $N_3$) is the number of squarefree integers $d$ (resp.
of $3$-rational fields $\Q(\sqrt d)$) in the interval:

\subsubsection{Real case for $K$} 
The mirror field of $K=\Q(\sqrt d)$ is $K^*=\Q(\sqrt {-3\,d})$ and 
we must have $d \equiv 3 \pmod 9$ when $3 \mid d$, otherwise
$3$ splits in $K(\mu_3^{})/K$ and ${\mathcal W}_K \ne 1$.

\smallskip
\footnotesize
\begin{verbatim}
{N=0;N3=0;for(d=1,10^3,if(core(d)!=d || Mod(d,9)==-3,next);N=N+1;
v=valuation(d,3);D=coredisc(-3^(1-2*v)*d);h=qfbclassno(D);
if(Mod(h,3)!=0,N3=N3+1));print(N3," ",N," ",N3/N+0.0)}
N3=315551, N=531938, N3/N=0.59321
\end{verbatim}

\normalsize\smallskip
If we restrict ourselves to integers $d \equiv 3 \pmod 9$,
we obtain $N_3=180717$, $N= 303961 $ and the
proportion $0.59454$. With the condition $3 \nmid d$
(corresponding to the non-ramification of $3$ in $K$) we obtain
the similar proportion $0.59185$.

\smallskip
This is consistent with the heuristics of Cohen--Martinet, about $3$-class 
groups (see \cite[Section 2, \S\,1.1 (b)]{CM} or \cite[\S\,9\,(b)]{CL}), giving
the probability $0.439874$ for $3 \mid \order \Cl_{k^*}$, whence $0.560126$ for 
the contrary.

\subsubsection{Imaginary case for $K$}
The program uses the computation of the ray class group of modulus $9$
(see Program I of \S\,\ref{prog2}):

\smallskip
\footnotesize
\begin{verbatim}
{N=0;N3=0;for(dd=1,10^6,d=-dd;if(core(d)!=d || Mod(d,9)==-3,next);N=N+1;
p=3;n=2;r=2;P=x^2-d;K=bnfinit(P,1);Kpn=bnrinit(K,p^n);
Hpn=component(component(Kpn,5),2);L=listcreate;e=component(matsize(Hpn),2);
R=0;for(k=1,e,c=component(Hpn,e-k+1);if(Mod(c,p)==0,R=R+1;
listinsert(L,p^valuation(c,p),1)));if(R==r,N3=N3+1));
print(N3," ",N," ",N3/N+0.0)}
N3=462125,N=531934,N3/N=0.868
\end{verbatim}

\normalsize\smallskip
This is in accordance with the probabilities about $3$-class groups 
in the real case giving the probability $0.841$ for $3 \nmid \order \Cl_{k^*}$
(see \cite[Section 2, \S\,1.2 (b)]{CM}).

\smallskip
Thus, there are approximatively $60 \%$ of $3$-rational real quadratic fields
and around $86 \%$ of $3$-rational imaginary quadratic fields.
So we do not need more precise information, but we must only assume
that $3$-rational quadratic fields are uniformly distributed whatever
the interval of range of the discriminants and that the probabilities are
around $0.60$ (resp. $0.86$) for our reasonnings. 
For more precise numerical results and heuristics, see \cite[\S\,5]{BR}, 
\cite[Section 5]{PV} and the very complete work on imaginary 
quadratic fields \cite{PS} and all $p\geq 2$.

\subsection{Analyse of Greenberg's conjecture}

We denote by $K_t$, $t \geq 1$, a compositum of one of the two forms:
$$\hbox{$K_t = \Q(\sqrt{d_1^{}}, \ldots , \sqrt{d_t^{}})$, with $3 \nmid d_i$ for all $i$,} $$ 
\centerline{$K_t=\Q(\sqrt{d_1^{}}, \ldots , \sqrt{d_{t-1}}, \sqrt{3 \cdot d_t^{}})$,
with $d_i \equiv 1 \pmod 3$ for all $i$.}

\medskip
From the above heuristics, we can say that for $t=1$ there are reasonably
infinitely many such $3$-rational $K_1$ with density $P_1$ less than 
$0.6$ or $0.86$.

\smallskip
Now take $t \geq 3$; since $K_t$ is $3$-rational if and only if its $2^t-1$
quadratic subfields are $3$-rational, the probability may be assumed to be
as follows, under the assumption that all the generators $d_1,\ldots,d_t$
are random under the necessary local conditions of 
Corollary \ref{congruences}:

\smallskip
\quad $\bullet$ If $K_t$ is real, the probability is:
$$P_t \approx 0.60^{2^{t}-1}$$
(giving $1.33 \cdot 10^{-7}$ for $t=5$, $1.06 \cdot 10^{-14}$ for $t=6$ and
$1.12 \cdot 10^{-227}$ for $t=10$). 

\smallskip
\quad $\bullet$ If $K_t$ is imaginary, we have $2^{t-1}-1$ real
quadratic subfields and $2^{t-1}$ imaginary ones, so that the probability is~:
$$P_t \approx 0.60^{2^{t-1}-1} \cdot 0.86^{2^{t-1}}$$
(giving $4.2 \cdot 10^{-5}$ for $t=5$, $1.06 \cdot 10^{-9}$ for $t=6$ and
$1.25 \cdot 10^{-147}$ for $t=10$).

\smallskip
This explains the great difficulties to find numerical examples, for $t>6$
with $p=3$, in a reasonable interval $[b, B]$ for $d_1, \ldots,d_t$.
Recall the imaginary example with $t=6$ given in \cite[\S\,4.2]{Gre}:
$\Q(\sqrt{-1}, \sqrt{13}, \sqrt{145}, \sqrt{209}, \sqrt{269}, \sqrt{373})$
(we shall give another similar example in \S\,\ref{imaginary}).

\smallskip
Naturally, when $p$ increases, it is easier to find examples with larger $t$
but the problem is similar. See \cite{BR} for some examples with $p\geq 5$
with quadratic fields (e.g., Table 1 for $\Q(\sqrt{2}, \sqrt{3}, \sqrt{11}, 
\sqrt{47}, \sqrt{97})$ with $p=5$) and cubic fields, then \cite{Gr2, Gr6, HZ}
for $p$-adic regulators. In the framwork of Borel--Cantelli classical heuristic, 
the conjecture of Greenberg may be true.

\section{Computations for compositum of quadratic fields}

\subsection{Research of real $3$-rational compositum of $5$ quadratic fields}
\label{real}

For practical reasons, we write two programs depending on the ramification of
$3$ in the fields $K=\Q(\sqrt{d_1}, \sqrt{d_2}, \sqrt{d_3}, \sqrt{d_4}, \sqrt{d_5})$.
The two programs verify, from the data, that ${\rm Gal}(K/\Q) \simeq (\Z/2\,\Z)^5$.

\subsubsection{Case where $3$ is possibly ramified} 
The program computes a list 
of integers $d>0$ such that $d \equiv 1 \pmod 3$ or $d \equiv 3 \pmod 9$
(see Corollary \ref{congruences}) and such that the $\Q(\sqrt{d})$ are $3$-rational
(then $3$ is possibly unramified in some solutions $K$, but in that case $d_i
\equiv 1 \pmod 3$ for all $i$).

\smallskip\footnotesize
\begin{verbatim}
{L3=listcreate;N3=0;b=1;B=300;for(d=b,B,
if(core(d)!=d || Mod(d,3)==-1 || Mod(d,9)==-3,next);v=valuation(d,3);
D=coredisc(-3^(1-2*v)*d);h=qfbclassno(D);if(Mod(h,3)!=0,N3=N3+1;
listinsert(L3,d,1)));for(k=1,10^7,a1=random(N3)+1;a2=random(N3)+1;
a3=random(N3)+1;a4=random(N3)+1;a5=random(N3)+1;
d1=component(L3,a1);d2=component(L3,a2);d3=component(L3,a3);
d4=component(L3,a4);d5=component(L3,a5);TT=0;
for(e1=0,1,for(e2=0,1,for(e3=0,1,for(e4=0,1,for(e5=0,1,
dd=d1^e1*d2^e2*d3^e3*d4^e4*d5^e5;dd=core(dd);
if(dd==1 & [e1,e2,e3,e4,e5]!=[0,0,0,0,0],TT=1;break(5)))))));
if(TT==0,T=0;for(e1=0,1,for(e2=0,1,for(e3=0,1,for(e4=0,1,for(e5=0,1,
dd=d1^e1*d2^e2*d3^e3*d4^e4*d5^e5;dd=core(dd);
if(Mod(dd,3)!=0,D=coredisc(-3*dd));if(Mod(dd,3)==0,D=coredisc(-dd/3));
h=qfbclassno(D);if(Mod(h,3)==0,T=1;break(5)))))));
if(T==0,print(d1," ",d2," ",d3," ",d4," ",d5))))}
\end{verbatim}

\normalsize\smallskip
One obtains the following distinct examples:

\smallskip
\footnotesize
\begin{verbatim}
[d1,d2,d3,d4,d5]=
[118,178,31,46,211],[274,291,66,118,262],[22,193,13,262,163],
[37,274,31,211,46],[31,130,166,129,246],[298,13,7,111,210],
[201,157,57,55,219],[255,282,165,298,118],[19,211,61,166,217],
[39,30,129,111,166],[187,246,39,145,31],[66,265,246,219,157],
[55,219,217,102,205],[193,262,163,10,127],[37,61,273,205,21],
[255,115,259,193,7],[31,187,145,246,39]
\end{verbatim}

\normalsize
\subsubsection{Case where $3$ is unramified}
In this case there is no condition on the $d_i \not \equiv 0 \pmod 3$.

\smallskip\footnotesize
\begin{verbatim}
{L3=listcreate;N3=0;b=2;B=300;for(d=b,B,if(core(d)!=d||Mod(d,3)==0,next);
D=coredisc(-3*d);h=qfbclassno(D);if(Mod(h,3)!=0,N3=N3+1;listinsert(L3,d,1)));
for(k=1,10^7,a1=random(N3)+1;a2=random(N3)+1;a3=random(N3)+1;a4=random(N3)+1;
a5=random(N3)+1;d1=component(L3,a1);d2=component(L3,a2);d3=component(L3,a3);
d4=component(L3,a4);d5=component(L3,a5);TT=0;
for(e1=0,1,for(e2=0,1,for(e3=0,1,for(e4=0,1,for(e5=0,1,
dd=d1^e1*d2^e2*d3^e3*d4^e4*d5^e5;dd=core(dd);
if(dd==1 & [e1,e2,e3,e4,e5]!=[0,0,0,0,0],TT=1;break(5)))))));
if(TT==0,T=0;for(e1=0,1,for(e2=0,1,for(e3=0,1,for(e4=0,1,for(e5=0,1,
dd=d1^e1*d2^e2*d3^e3*d4^e4*d5^e5;dd=core(dd);D=coredisc(-3*dd);
h=qfbclassno(D);if(Mod(h,3)==0,T=1;break(5)))))));
if(T==0,print(d1," ",d2," ",d3," ",d4," ",d5))))}
\end{verbatim}

\normalsize\smallskip
One obtains the following examples:

\smallskip\footnotesize
\begin{verbatim}
[d1,d2,d3,d4,d5]=
[91,230,209,194,221],[149,335,301,55,31],[145,157,230,209,194],
[35,193,301,149,263],[226,239,130,158,259],[190,259,143,70,290],
[266,59,17,10,70],[194,11,283,187,31],[107,227,34,130,230],
[13,17,215,31,61],[145,133,218,34,230],[22,215,221,31,161],
[86,149,170,146,145],[158,259,226,146,47],[146,26,269,166,190],
[46,170,38,133,187],[190,119,97,14,263],[203,227,221,194,143],
[239,89,262,53,166],[13,262,193,163,286]
\end{verbatim}

\normalsize\smallskip
In a larger interval, the examples become more rare
since the number of discriminants decrease; between
$b=10^3$ and $B=10^3+350$ we get the solutions (when $3$ can ramify):

\smallskip
\footnotesize
\begin{verbatim}
[d1,d2,d3,d4,d5]=
[1245,1303,1218,1291,1123],[1177,1003,1309,1173,1054],
[1321,1065,1173,1231,1207],[1155,1326,1111,1105,1093],
[1173,1281,1254,1327,1209],[1030,1174,1177,1218,1015]
\end{verbatim}

\normalsize
\subsubsection{Direct verifications}
To verify, one may compute directly the $3$-class number and the $3$-adic 
logarithm of the fundamental unit $\varepsilon$ of $\Q(\sqrt d)$ for each 
quadratic subfield of $K$; the computation of the normalized regulator
(using \cite[Proposition 5.2]{Gr2}) is equivalent to that of
$\alpha := \Frac{\varepsilon^q-1}{3}$ when $3$ is unramified 
(where $q=2$ or $8$), and $\alpha := \Frac{\varepsilon^2-1}{\sqrt d}$ 
when $3 \mid d$, and $\alpha$ must be a $3$-adic unit (which can 
be seen taking its norm):

\smallskip
\footnotesize
\begin{verbatim}
{L3=[110,170,161,38,14];d1=component(L3,1);d2=component(L3,2);
d3=component(L3,3);d4=component(L3,4);d5=component(L3,5);T=0;
for(e1=0,1,for(e2=0,1,for(e3=0,1,for(e4=0,1,for(e5=0,1,
dd=d1^e1*d2^e2*d3^e3*d4^e4*d5^e5;dd=core(dd);if(dd==1,next);
D=coredisc(dd);h=qfbclassno(D);eps=quadunit(D);q=8;if(Mod(dd,3)!=-1,q=2);
E=eps^q-1;if(Mod(D,3)!=0,No=norm(E)/9);if(Mod(D,3)==0,No=norm(E)/3);
No=Mod(No,3);print(dd," ",Mod(h,9)," ",No))))))}
\end{verbatim}

\normalsize\smallskip\noindent
giving, for $\Q(\sqrt{110}, \sqrt{170}, \sqrt{161}, \sqrt{38}, \sqrt{14})$
(computation of $h$ modulo $9$ for information instead of modulo $3$):

\smallskip
\footnotesize
\begin{verbatim}
d       h mod 9      Regulator       d       h mod 9      Regulator
14      Mod(1,9)     Mod(1,3)        385     Mod(2,9)     Mod(2,3)
38      Mod(1,9)     Mod(1,3)        1045    Mod(4,9)     Mod(2,3)
133     Mod(1,9)     Mod(2,3)        14630   Mod(8,9)     Mod(1,3)
161     Mod(1,9)     Mod(1,3)        17710   Mod(8,9)     Mod(2,3)
46      Mod(1,9)     Mod(2,3)        1265    Mod(2,9)     Mod(1,3)
6118    Mod(4,9)     Mod(2,3)        168245  Mod(8,9)     Mod(1,3)
437     Mod(1,9)     Mod(1,3)        48070   Mod(8,9)     Mod(2,3)
170     Mod(4,9)     Mod(1,3)        187     Mod(2,9)     Mod(2,3)
595     Mod(4,9)     Mod(2,3)        2618    Mod(4,9)     Mod(1,3)
1615    Mod(4,9)     Mod(2,3)        7106    Mod(4,9)     Mod(1,3)
22610   Mod(8,9)     Mod(1,3)        24871   Mod(8,9)     Mod(2,3)
27370   Mod(7,9)     Mod(2,3)        30107   Mod(8,9)     Mod(1,3)
1955    Mod(4,9)     Mod(1,3)        8602    Mod(4,9)     Mod(2,3)
260015  Mod(5,9)     Mod(1,3)        1144066 Mod(7,9)     Mod(2,3)
74290   Mod(8,9)     Mod(2,3)        81719   Mod(8,9)     Mod(1,3)
110     Mod(2,9)     Mod(1,3)
\end{verbatim}

\normalsize
\subsection{Research of imaginary $3$-rational compositum}
\label{imaginary}
The programs written in \S\,\ref{real} are valid for any interval 
$[b, B]$ in $\Z$ and we find, as expected, more solutions that in the real case
for $t=5$. 

\smallskip
Give only an example since the calculations are
different for imaginary quadratic subfields.

\smallskip
Let $K=\Q(\sqrt{-1}, \sqrt{7}, \sqrt{10}, \sqrt{13}, \sqrt{37})$; then
a monic polynomial defining $K$ is:

\smallskip\footnotesize
\begin{verbatim}
x^32-1056*x^30+480032*x^28-124184704*x^26+20397674176*x^24
-2244202678784*x^22+169874210289152*x^20-8952078632101888*x^18
+329969412292171264*x^16-8547361273484173312*x^14
+156988254584490745856*x^12-2053956312746026434560*x^10
+19066991321006131953664*x^8-123357558863823312388096*x^6
+535739176635907164471296*x^4-1443775880343438717616128*x^2
+1984177860024815997485056
\end{verbatim}

\normalsize\smallskip\noindent
and the Program I of \S\,\ref{prog2} confirms the $3$-rationality.

\medskip
We also find a new example with $t=6$:
$$K=\Q(\sqrt{-2}, \sqrt{-5}, \sqrt{7}, \sqrt{17}, \sqrt{-19}, \sqrt{59}). $$
Such examples are very rare and it is probably hopeless to find a
numerical example with $t=7$ in a reasonable interval of discriminants.

\smallskip
To verify the $3$-rationalities, we use the program of the above subsection, 
noting that for imaginary quadratic subfields $k$ there is no unit $\varepsilon$ 
and that its $3$-class group may be non-trivial, but in this case the $3$-Hilbert 
class field must be contained in $\wt k$ which gives interesting 
examples of such phenomena; indeed, this fact is not obvious without 
explicit knowledge of generators of the $3$-classes 
and we can use directly the Program I of $p$-rationality for each $k$ such that 
$h\equiv 0 \pmod 3$. 

\smallskip
\footnotesize
\begin{verbatim}
{L3=[70,59,-118,-19,17,-14];d1=component(L3,1);d2=component(L3,2);
d3=component(L3,3);d4=component(L3,4);d5=component(L3,5);d6=component(L3,6);
T=0;for(e1=0,1,for(e2=0,1,for(e3=0,1,for(e4=0,1,for(e5=0,1,for(e6=0,1,
dd=d1^e1*d2^e2*d3^e3*d4^e4*d5^e5*d6^e6;dd=core(dd);
if(dd==1,next);D=coredisc(dd);h=qfbclassno(D);
if(D>0,eps=quadunit(D);q=8;if(Mod(dd,3)!=-1,q=2);E=eps^q-1;
if(Mod(D,3)!=0,No=norm(E)/9);if(Mod(D,3)==0,No=norm(E)/3);No=Mod(No,9));
if(D<0,No=X);print(dd," ",3^valuation(h,3)," ",No)))))))}
\end{verbatim}

\normalsize\smallskip

\footnotesize
\begin{verbatim}
 d        h     No                 d         h     No
-14       1     X                  -5        1     X
17        1     Mod(1,3)           1190      1     Mod(1,3)
-238      1     X                  -85       1     X
-19       1     X                  -1330     3     X     
266       1     Mod(1,3)           95        1     Mod(1,3)
-323      1     X                  -22610    1     X
4522      1     Mod(2,3)           1615      1     Mod(2,3)
-118      3     X                  -2065     3     X    
413       1     Mod(1,3)           590       1     Mod(1,3)
-2006     3     X                  -35105    1     X
7021      1     Mod(2,3)           10030     1     Mod(2,3)
2242      1     Mod(2,3)           39235     1     Mod(2,3)
-7847     1     X                  -11210    1     X
38114     1     Mod(1,3)           666995    1     Mod(1,3)
-133399   3     X                  -190570   1     X
59        1     Mod(1,3)           4130      1     Mod(1,3)
-826      3     X                  -295      1     X
1003      1     Mod(2,3)           70210     1     Mod(2,3)
-14042    1     X                  -5015     1     X
-1121     1     X                  -78470    9     X   
15694     1     Mod(2,3)           5605      1     Mod(2,3)
-19057    1     X                  -1333990  1     X
266798    1     Mod(1,3)           95285     1     Mod(1,3)
-2        1     X                  -35       1     X
7         1     Mod(2,3)           10        1     Mod(2,3)
-34       1     X                  -595      1     X
119       1     Mod(1,3)           170       1     Mod(1,3)
38        1     Mod(1,3)           665       1     Mod(1,3)
-133      1     X                  -190      1     X
646       1     Mod(2,3)           11305     1     Mod(2,3)
-2261     9     X                  -3230     9     X   
70        1     Mod(2,3)  
\end{verbatim}

\normalsize\smallskip
From this table, we deduce that the $3$-class group of $K$
is isomorphic to $(\Z/9\,\Z)^3 \times (\Z/3\,\Z)^6$ and that the $3$-Hilbert 
class field of $K$ is contained in the compositum of the $16$ 
independent $\Z_3$-extensions of $K$ distinct from the cyclotomic one
(i.e., the compositum of the $16$ ``anti-cyclotomic'' ones).

\section{Number fields $p$-rational for all $p \geq 2$}

As soon as there exist units of infinite order in $K$,
the $p$-rationality is a very difficult question and {\it a fortiori}
the existence of such fields, $p$-rational for all $p$. 
So it remains to consider the fields
such that $r_1+r_2-1=0$, which characterizes $K=\Q$ (which is
$p$-rational for all $p$) and the imaginary quadratic fields for which
we recall some history.

\subsection{The $p$-rationality of imaginary quadratic fields}\label{all}
After a work by Onabe \cite{O} (with a correction to be made in the case $p=2$), 
the subject was studied by Angelakis and Stevenhagen 
\cite[Theorem 4.4]{AS} in a different 
(but essentially equivalent) setting. Indeed, it is 
immediate to see that to have a {\it minimal absolute abelian Galois 
group, isomorphic to $\widehat \Z^2 \times \prd_{n\geq 1} \Z/n\,\Z$}, 
is equivalent, for the  imaginary quadratic field $K$, to be 
$p$-rational for all $p$, what we have explained and generalized 
for any number field in \cite[\S\,1.1]{Gr5}.

\medskip
Let $K=\Q(\sqrt {-d})$ be an imaginary quadratic field and let $p$ be any 
prime number. Then $K$ is $p$-rational if and only if ${\mathcal W}_K=1$ 
and the $p$-Hilbert class field $H_K$ is contained in $\wt K$.

\smallskip
For $p=2$, the conditions are given in Example \ref{p=2} 
about complex and real abelian fields of degree a power of $2$. 
For $p=3$, ${\mathcal W}_K \simeq \Z/3\Z$ for $-d\equiv -3 \pmod 9$
and $-d \ne -3$.
Then for $p>3$, we get ${\mathcal W}_K=1$, but there is no immediate
criterion for the inclusion $H_K \subset \wt K$ and we must use for 
instance the usual numerical computations of \S\,\ref{prog2}.
We note that in \cite[Table 1, \S\,7]{AS} one has many examples
of non-$p$-rational fields $K$ whose $p$-class group is of order $p$ 
($2 \leq p \leq 97$), in complete agreement with our PARI Program I.

\smallskip
The Conjecture 7.1 of \cite{AS} may be translated into the similar one~:
$$\hbox{\it There are infinitely many imaginary quadratic fields, $p$-rational for all $p$.}$$
Indeed, for $K$ fixed, the class group is in general cyclic (from Cohen--Lenstra--Martinet
heuristics) and for each $p$-class group the inclusion of the $p$-Hilbert class field
in $\wt K$ depends on $p$-adic values of logarithms of ideals generating the
$p$-classes (see Section \ref{def} and Remark \ref{remalog} below), 
whose probabilities are without mystery;
but no direction of proof is known and the fact that the set of prime divisors 
of the class numbers increases as $d \to \infty$ implies some rarefaction 
of such fields:

\smallskip
In $[b, B]=[1,10^6]$ the proportion is $0.0766146$
($46576$ solutions for $607926$ fields), in $[b, B]=[10^6,10^6+10^5]$ it is
$0.0697357$, in $[10^9;10^9+10^5]$ it is $0.0454172$, and in
$[10^{11}, 10^{11}+10^5]$ it is $0.0379389$ ($2306$ solutions for $60782$
fields).

\medskip
The following program gives very quickly, in any interval $[b, B]$, the set 
of imaginary quadratic fields which are $p$-rational for all $p \geq 2$; 
for this, it verifies the $p$-rationalities for $p=2$, $3$, and when $p$
divides the class number:

\smallskip
\footnotesize
\begin{verbatim}
{b=1;B=150;Lrat=listcreate;m=0;for(d=b,B,if(core(d)!=d,next);P=x^2+d;
K=bnfinit(P,1);h=component(component(component(K,8),1),1);
hh=component(factor(6*h),1);t=component(matsize(hh),1);
for(k=1,t,p=component(hh,k);n=2;if(p==2,n=3);Kpn=bnrinit(K,p^n);
Hpn=component(component(Kpn,5),2);e=component(matsize(Hpn),2);
R=0;for(j=1,e,c=component(Hpn,e-j+1);if(Mod(c,p)==0,R=R+1));
T=0;if(R>2,T=1;break));if(T==0,m=m+1;listinsert(Lrat,-d,m)));print(Lrat)}

[-1,-2,-3,-5,-6,-10,-11,-13,-19,-22,-26,-29,-37,-38,-43,-53,-58,-59,-61,
-67,-74,-83,-86,-101,-106,-109,-118,-122,-131,-134,-139,-149] ...
[-1000058,-1000099,-1000117,-1000133,-1000138,-1000166,-1000211,-1000213,
-1000214,-1000253,-1000291,-1000333,-1000357,-1000358,-1000381,-1000394 ...
[-10000000019,-10000000058,-10000000061,-10000000069,-10000000118,
-10000000147,-10000000198,-10000000277,-10000000282,-10000000358 ...
[-123456789062,-123456789094,-123456789133,-123456789194,-123456789322,
-123456789403,-123456789419,-123456789451,-123456789563,-123456789587 ...
\end{verbatim}

\normalsize
\begin{remark} \label{remalog}
Let $K$ be an imaginary quadratic field and $p>3$; the ${\rm Log}$ 
function is nothing else than the usual logarithm. 
Then, using the property relying on the characterization 
``${\mathcal T}_K = {\rm Ker}({\rm Log})$'', {\it $K$ is not $p$-rational} as soon as
there exists an ideal ${\mathfrak a}$, whose class is of order $p^e \ne 1$, such that 
${\rm log} ({\mathfrak a}) \in {\rm log}(U_K)$ which is equivalent to
$\alpha = \xi \cdot u^{p^e}$, where ${\mathfrak a}^{p^e}=(\alpha)$
and $\xi \in \bigoplus_{{\mathfrak p} \mid p} \mu_{p^f-1}^{}$
where $f \in \{1,2\}$ is the residue degree of $p$.

\smallskip
For instance, in $K=\Q(\sqrt{-383})$, for $p=17$, the $p$-class group is
generated by the class of ${\mathfrak l} \mid 2$ such that
${\mathfrak l}^{17} = (\alpha)$ where $\alpha = \frac{711+7\,\sqrt{-383}}{2}$
and ${\rm log}(\alpha) \equiv 0 \pmod {17^2}$; since ${\rm log}(U_K) =
17\, (\Z_{17}\, \oplus\, \Z_{17} \,\sqrt{-383})$, $K$ is not $17$-rational.

\smallskip
So it should be easy to elaborate PARI programs using this point of view.
\end{remark}

\subsection{Application to a conjecture of Hajir--Maire}

In \cite[Conjecture 0.2]{HM} is proposed the following 
conjecture as a sufficient condition to construct
extensions of number fields whose Galois group is a suitable 
uniform pro-$p$-group with arbitrary large Iwasawa $\mu$-invariant:

\smallskip
{\it Given a prime $p$ and an integer $m \geq 1$, coprime to $p$, there exist a
totally imaginary field $K_0$ and a degree $m$ cyclic extension 
$K/K_0$ such that $K$ is $p$-rational.}

\smallskip
and it is conjectured that the statement is true taking for $K_0$ an
imaginary $p$-rational quadratic field \cite[Conjecture 4.16]{HM}.

\subsubsection{Analysis of the conjecture}
Under Leopoldt's conjecture, we have seen that $K_0$ must be 
itself $p$-rational, so the computations in \S\,\ref{all}
show that, in practice, we may fix $K_0$ to be any imaginary 
quadratic field, $p$-rational for all $p$, since they are very numerous. 
Then one has only to choose $p$ and $m$ and find a degree $m$ cyclic 
$p$-rational extension of $K_0$.

\smallskip
Consider such a field $K_0=\Q(\sqrt{-d})$;
we know that its absolute abelian Galois group
${\rm Gal}(K^{\rm ab}/K)$ is isomorphic to
$\, \widehat \Z^2 \times \prd_{n\geq 1} \Z/n\,\Z$,
giving a countable basis of cyclic extensions
of degree $m$, coprime to $p$, and the defect of the conjecture
would say that {\it any} degree $m$ cyclic extension $K$ of $K_0$ 
is non-$p$-rational (this coming essentialy from the $p$-class group 
and/or the normalized $p$-adic regulator of $K$, since 
one can easily realize ${\mathcal W}_K=1$ for infinitely many $K$). 

\smallskip
The Cohen--Lenstra--Martinet heuristics \cite{CL, CM} show that 
infinitely many cyclic extensions of degree $m$ may have trivial 
$p$-class group, so that we must focus on the 
case of $p$-adic properties of units of such fields.

\smallskip
For simplicity, we restrict ourselves to extensions 
$K/K_0$ unramified at $p$ since in the non-$p$-part of
$K^{\rm ab}/K$, the inertia groups at the $p$-places are finite
and it remains infinitely many choices.

\smallskip
Let $G = {\rm Gal}(K/K_0) \simeq \Z/m\Z$; then 
$E_K/\mu_K^{}$ (still denoted by abuse $E_K$) is a $G$-module 
of $\Z$-rank $m-1$ such that ${\rm N}_{K/K_0}(E_K)=1$,
where the norm ${\rm N}_{K/K_0}$ may also be understood as 
the ``algebraic norm'' in $\Z[G]$.
It is well known, from generalized Herbrand's theorem 
(e.g., \cite[Lemma I.3.6]{Gr1}), that there exists a ``Minkowski unit'' 
$\eta \in E_K$ generating a $\Z[G]$-module $E'_K$ of prime to 
$p$ index in $E_K$; so $E'_K$ is a monogenic 
$\Z[G]/({\rm N}_{K/K_0})$-module.

\smallskip
Thus ${\mathcal R}_K$ is $p$-adically equivalent to a
Frobenius determinant, product of components, indexed with
the $p$-adic characters $\theta \ne 1$ of $G$, of the form
${\rm Reg}_p^\theta(\eta) = \prod_{\varphi \div \theta}{\rm Reg}_p^\varphi(\eta)$
where $\varphi$ runs trough a set of $\Q_p$-conjugates of some 
irreducible characters of $G$.

\smallskip
The field of values of $\varphi$ only depend on the 
rational character $\chi$ such that $\varphi \mid \theta \mid \chi$
and is denoted $C_\chi$; thus the ${\rm Reg}_p^\varphi(\eta)$ are 
$C_\chi$-linear combinations
of terms of the form $\frac{1}{p}\,{\rm log}(\eta^\sigma)$, $\sigma \in G$,
and the probabilities of ${\rm Reg}_p^\theta(\eta) \equiv 0 \pmod p$
depend on the residue degree $f$ of $p$ in $C_\chi$ and are conjecturaly
at most in $\frac{O(1)}{p^f}$ \cite[Section 4, \S\,4.1]{Gr3}.
Then each case ($p$ being fixed) strongly depends 
on the selected value of $m$, but the probability of non-$p$-rationality 
is a constant $\pi(p,m)$ whatever the choice of $K$.

\smallskip
This gives an incredible probability to have ${\mathcal R}_K \equiv 0 \pmod p$ 
for all these cyclic fields of degree $m$; moreover this concerns 
$K_0$ fixed among (conjecturally) infinitely many fields.

\begin{remark} The above reasoning is valid if we take $K$ as the compositum
of $K_0$ with a real cyclic extension $K_1$ of degree $m \not\equiv 0 \pmod p$ 
of $\Q$. If for instance $m$ is odd, $K$ is $p$-rational if and only if 
$K_1$ is $p$-rational and $\Cl_K^\infty=1$ since 
${\mathcal R}_K={\mathcal R}_{K_1}=1$.
So in practice, varying $K_1$, we are certain to obtain, numerically, a solution.
To prove that it is always possible is not yet accessible since we do not have 
any example, for each $m$, of a field $K_1$, $p$-rational for all $p$
(which is not sufficient to get the $p$-rationality of $K$).

\smallskip
We shall illustrate the two contexts ($K/\Q$ non-abelian and $K=K_0 K_1$).
\end{remark}

\subsubsection{Numerical examples} 

(i) Take $K_0=\Q(\sqrt{-3})$
and $m=6$ to define $K$ by means of Kummer theory. Put 
$K=K_0(\sqrt[6]{u+v\,\sqrt{-3}})$, $u, v \in \Z$; then $K$ is defined by the
polynomial $P = x^{12}-2 u\,x^6+u^2+3v^2$ and $r:=r_2+1=7$. The program
verifies that $P$ is irreducible since $u, v$ are random.
The $p$-rationality is tested for $5 \leq p \leq 100$ and in most cases
the $p$-rationality holds giving many possibilities to illustrate the conjecture.

\smallskip
We give the list of exceptions obtained in an execution of the program 
(among about one hundred $u+v\,\sqrt{-3}$):

\smallskip
\footnotesize
\begin{verbatim}
{n=2;r=7;b=5;B=100;for(N=1,100,u=random(10^3);v=random(10^3);
P=x^12-2*u*x^6+u^2+3*v^2;if(polisirreducible(P)!=1,next);K=bnfinit(P,1);
forprime(p=b,B,Kpn=bnrinit(K,p^n);Hpn=component(component(Kpn,5),2);
L=listcreate;e=component(matsize(Hpn),2);R=0;for(k=1,e,
c=component(Hpn,e-k+1);if(Mod(c,p)==0,R=R+1;
listinsert(L,p^valuation(c,p),1)));if(R>r,
print(u," ",v," K is not ",p,"-rational"))))}
 u      v                               u      v
705    960     K is not 7-rational     351    750     K is not 37-rational
705    960     K is not 19-rational    286    682     K is not 7-rational
497    401     K is not 61-rational    60     298     K is not 7-rational
663    465     K is not 7-rational     56     789     K is not 7-rational
593    796     K is not 5-rational     677    538     K is not 7-rational
75     38      K is not 7-rational     884    54      K is not 11-rational
75     38      K is not 11-rational    646    177     K is not 7-rational
351    750     K is not 7-rational     25     130     K is not 13-rational
\end{verbatim}

\normalsize\smallskip
(ii) For the compositum $K$ of $K_0=\Q(\sqrt {-1})$ with the cyclic
extension of $\Q$ of degree $5$ and conductor $11$, we obtain that
$K$ is not $p$-rational for $p=761$, up to $10^4$.
To have $\Cl_K \ne 1$, it is necessary that ${\rm N}_{K/K_0}(\Cl_K)=
{\rm N}_{K/K_1}(\Cl_K)=1$, which explains the rarity of the 
non-$p$-rationalities:

\smallskip\footnotesize
\begin{verbatim}
{P=polcompositum(polsubcyclo(11,5),x^2+1);P=component(P,1);K=bnfinit(P,1);
b=2;B=10^4;r=component(component(component(K,7),2),2)+1;
forprime(p=b,B,n=2;if(p==2,n=3); Kpn=bnrinit(K,p^n);
Hpn=component(component(Kpn,5),2);e=component(matsize(Hpn),2); 
R=0;for(k=1,e,c=component(Hpn,e-k+1);if(Mod(c,p)==0,R=R+1));
if(R>r,print("rk(T)=",R-r," K is not ",p,"-rational")))}
\end{verbatim}

\normalsize
\subsection{Incomplete $p$-rationality}  As suggested by Hajir and Maire, 
some cases of ``incomplete $p$-ramification'' may be useful for some theoretical 
aspects when $p$ splits in $K$ in more than one prime ideal;
for instance, in the case of imaginary quadratic fields in which $p$ splits, 
one may consider the ray class field $K({\mathfrak p})$ of modulus 
${\mathfrak p} \mid p$ and class groups formulas in $K({\mathfrak p})$ 
with elliptic units, in the framework of the work of \cite{K} for one of the 
nine principal fields~$K$.\par

Then one can hope that the groups ${\mathcal T}^{(\ell)}_{K({\mathfrak p})}$ 
(for $\ell$-ramification theory with any prime $\ell$, 
the base field being $K({\mathfrak p})$ fixed, insted of $K$) 
can be interpreted with these analytic 
formulas and the corresponding question of the $\ell$-rationalities 
of $K({\mathfrak p})$ may be of some interest to generalize the 
classical abelian context.

\subsubsection{General definition of incomplete $p$-rationality}
Consider the general situation of a number field $K$ with any prime $p \geq 2$.
Let $P$ be a subset of the set of $p$-places of $K$ and let $H_K^{{\sst P}\rm ram}$
be the maximal abelian pro-$p$-extension of $K$, unramified outside~$P$; 
the corresponding formula is available in \cite[Theorems III.2.5 \& III.2.6]{Gr1}, 
stated in the ordinary sense, and gives for the torsion group
${\mathcal T}_{K}^{\sst P}$ of ${\rm Gal}(H_K^{{\sst P}\rm ram}/K)$:
\begin{equation}\label{partial}
\order {\mathcal T}_{K}^{\sst P} = \order {\rm tor}_{\Z_p}
\Big( \plus_{{\mathfrak p} \in P} U_{\mathfrak p}\Big / \ov {E_K}^{\sst P} \Big) \cdot
\big [H_K : H_K \cap \wt {K\ }^{\sst P} \big],
\end{equation}
where $\ov {E_K}^{\sst P}$ is the closure of the image of $E_K$ in 
$\bigoplus_{{\mathfrak p} \in P} U_{\mathfrak p}$ 
(i.e., the projection of $\ov E_K$ in this product over $P$) and
$\wt {K\ }^{\sst P}$ the compositum of the $\Z_p$-extensions 
contained in $H_K^{{\sst P}\rm ram}$.\par

If the second factor $\big [H_K : H_K \cap \wt {K\ }^{\sst P}\big ]$ may be controled, 
and is trivial in most cases, the first one is more tricky since
$\bigoplus_{{\mathfrak p} \in P} U_{\mathfrak p}$ is not a Galois module, 
but the following definition makes sense (under the Leopoldt conjecture):

\medskip
{\it Let $p$ be a prime number and let $P$ be a subset of the set of $p$-places of $K$; 
the field $K$ is said to be $P$-rational if ${\mathcal T}_{K}^{\sst P}=1$.}

\medskip
We note that Theorem \ref{thmfond} \& Corollary \ref{corofond} are still valid for 
a test of any incomplete $p$-rationality, using \cite[Theorem I.4.5 \& 
Corollary I.4.5.4]{Gr1} for a more general support $P$: the value of $n_0$ is
the same assuming $e_{\mathfrak p}=1$ for all ${\mathfrak p} \in P$ 
(otherwise the suitable modulus is 
$\prod_{{\mathfrak p} \in P} \,{\mathfrak p}^{e_{\mathfrak p} n_0}$).

\subsubsection{$\{{\mathfrak p}\}$-rationality for imaginary quadratic fields}
For an imaginary quadratic fields with $p$ splitted into
${\mathfrak p}\,{\mathfrak p}'$ and $P=\{{\mathfrak p}\}$
one gets:

\smallskip
\centerline{$\order {\mathcal T}_{K}^{\{{\mathfrak p}\}} = \order
\big (\mu_{K_{\mathfrak p}}^{}  / \mu_K^{}  \big)\cdot
[H_K : H_K \cap \wt {K\ }^{\{{\mathfrak p}\}}]$,}

where $\wt {K\ }^{\{{\mathfrak p}\}}$ is a $\Z_p$-extension. 

\smallskip
The following program gives the non-$\{{\mathfrak p}\}$-rationality
for quadratic fields $\Q(\sqrt {-d})$, $d \in [b, B]$ and $p \in [bp, Bp]$:

\smallskip
\footnotesize
\begin{verbatim}
{b=1;B=10^3;bp=2;Bp=10^4;for(dd=b,B,if(core(dd)!=dd,next);
d=-dd;P=x^2-d;K=bnfinit(P,1);forprime(p=bp,Bp,if(kronecker(d,p)!=1,next);
n=2;if(p==2,n=3);p1=component(component(idealfactor(K,p),1),1);
pn=idealpow(K,p1,n);Kpn=bnrinit(K,pn);Hpn=component(component(Kpn,5),2);
L=listcreate;e=component(matsize(Hpn),2);R=0;for(k=1,e,
c=component(Hpn,e-k+1);if(Mod(c,p)==0,R=R+1;
listinsert(L,p^valuation(c,p),1)));if(R>1,
print("d=",d," rk(T)=",R-1," K is not ",p1,"-rational ",L))))}
\end{verbatim}

\normalsize\smallskip
We get the following non-$\{{\mathfrak p}\}$-rational fields (in the above intervals):

\medskip
$d \in \{-33,-57,-65,-105,-119,-129,-145,-161,-177,-185,\ldots\,$

\hfill $\ldots, -935,-943,-959,-969,-985,-993\}$, for ${\mathfrak p} \mid 2$, \par

\smallskip
$d \in \{-107,-302,
-362,-419,-503,-509,-533,-602,-617,-713$, 

\hfill $-863,-974\}$, for ${\mathfrak p} \mid 3$,\par

\smallskip
$d \in \{-166,-439,-449,-479,-601,-611,-739, -761,-874\}$, for ${\mathfrak p} \mid 5$,\par

\smallskip
$d \in \{-374,-530,-794,-831,-859,-894\}$, for ${\mathfrak p} \mid 7$,\par

\smallskip
$d=-758$, for ${\mathfrak p} \mid 11$,
$d \in \{-458,-998\}$, for ${\mathfrak p} \mid 13$,
$d=-383$, for ${\mathfrak p} \mid 17$.

\begin{remarks} (i) In an imaginary quadratic field in which $p\ne 2$ splits, 
the $p$-rationality is equivalent to the 
$\{{\mathfrak p}\}$-rationality: indeed, one verifies that 
$\wt {K\ }^{\{{\mathfrak p}\}}  \wt {K\ }^{\{{\mathfrak p}'\}} \! = {\wt K}$,
$\ \wt {K\ }^{\{{\mathfrak p}\}} \!\cap \wt {K\ }^{\{{\mathfrak p}'\}} \!= {\wt K} \!\cap H_K$
(contained in the ``anti-cyclotomic'' $\Z_p$-extension of $K$); thus the 
$p$-rationality implies the $\{{\mathfrak p}\}$ and $\{{\mathfrak p}'\}$-rationalities,
the converse being obvious.

\smallskip
For instance, for $p=3$, $K=\Q(\sqrt{-491})$ where
$\Cl_K \simeq \Z/9\,\Z$, we have 
$H_K = \wt {K\ }^{\{{\mathfrak p}\}} \cap \wt {K\ }^{\{{\mathfrak p}'\}}$
(whence $3$-rationality); the class of ${\mathfrak q} \mid 11$ is of order $9$
and ${\mathfrak q}^9=\big (\frac{1}{2} (95595 + 773 \sqrt {-491})\big)$
with ${\rm Log}\big(\frac{1}{2} (95595 + 773 \sqrt {-491})\big) \equiv 3\,\sqrt{-491} \pmod {27}$
giving again all the rationalities as expected.

\smallskip
So we may use the above program computing the ray class 
group modulo ${\mathfrak p}^{n_0}$ which is much faster than 
the program for the ray class group modulo~$p^{n_0}$.

\smallskip
(ii) In the case $p=2$ splitted in $K$, besides the use of Program I (\S\,\ref{prog2}),
the computations may be handled with the use of the 
$2$-adic logarithm as in \cite[Example III.5.2.2]{Gr1} for $K=\Q(\sqrt{-15})$, showing 
its $\{{\mathfrak p}\}$-rationality despite the fact that $K$ is not $2$-rational.

\smallskip
(iii) Let $K/\Q$ be a Galois extension and $P$ as above; let $\Sigma$ be
the set of prime ideals ${\mathfrak p} \mid p$, ${\mathfrak p} \notin P$.
The field $H_K^{{\sst P}\rm ram}$ is the subfield of $H_K^{\rm pr}$
fixed by the image of $U_K^{\sst\Sigma} := 
\bigoplus_{{\mathfrak p} \in \Sigma} U_{\mathfrak p}^1$ 
(corresponding,  by class field theory, to the subgroup 
generated by the inertia groups) in $U_K/ \ov E_K$, thus 
$U_K^{\sst\Sigma} \cdot \ov E_K / \ov E_K \simeq
 U_K^{\sst\Sigma}/U_K^{\sst\Sigma} \cap \ov E_K$ (then 
${\rm Gal}(H_K^{{\sst P}\rm ram}/H_K) \simeq 
(U_K /\ov E_K) \big / (U_K^{\sst\Sigma} \cdot \ov E_K / \ov E_K)$
is isomorphic to $U_K^{\sst P}/ \ov{E_K}^{\sst P}$ as expected 
for the formula \eqref{partial} giving $\order {\mathcal T}_{K}^{\sst P}$).

\smallskip
We have given generalization of Leopoldt conjecture in this
incomplete framework and conjectured a formula for the
$\Z_p$-rank of $U_K^{\sst\Sigma} \cap \ov E_K$ which is the set of 
$\ov \varepsilon \in \ov E_K$ in $U_K$ whose component on 
$U_K^{\sst P}$ is trivial (see \cite[Strong $p$-adic conjecture, III\,(f), 
Remarks 4.11.2, 4.11.4, 4.12.1, 4.12.3]{Gr1} for detailed statements; 
a $2013$ arXiv publication, by Dawn C. Nelson, discovers again the 
same kind of results: \url{https://arxiv.org/abs/1308.4637}).
\end{remarks}

\section{Conclusion}

\baselineskip=12.0pt
As we have seen, Galois number fields $K$ containing units of infinite
order are in general $p$-rational for ``almost all $p \gg 0$''
despite the fact that we have no information on ${\mathcal R}_K$ 
modulo $p$ for $p \to \infty$.
More precisely, probabilities are given (in an heuristic approach) 
by means of $p$-adic representation theory and the nature of 
$p$-adic characters of ${\rm Gal}(K/\Q)$, via factorization of 
Frobenius determinants (see \cite[Heuristique Principale, 
\S\,4.2.2]{Gr3} about the more general question of an algebraic number). 

\smallskip
The case of small primes $p$ is different because of local 
$p$th roots of unity and $p$-class groups. For instance, 
the invariant ${\mathcal W}_K = \big(\bigoplus_{{\mathfrak p} \mid p} 
\mu_{K_{\mathfrak p}}^{} \big)\big / \mu_K^{}$,
depending on the splitting of $p$ in $K(\mu^{}_p)/\Q$,
may be non-trivial (e.g., the most common case $p=2$).

\smallskip
For a Galois number field, with sufficiently ramified primes 
and $p \mid [K : \Q]$, the rank of the $p$-class group may be 
an obstruction to the $p$-rationality because of ``genera theory''
in $K/\Q$; the case where $p \nmid [K : \Q]$ leads also to an obstruction
as soon as ${\rm rk}_p(\Cl_K) > r = r_2+1$ and this fact is rather strange
in a probabilistic point of view. If $K$ is real and if $p \notdiv [K : \Q]$
then $\Cl_K^\infty = \Cl_K$ since $\wt K$
is the cyclotomic $\Z_p$-extension, totally ramified at $p$.

\smallskip
However, for $K$ fixed and $p \gg 0$, one gets $\Cl_K^\infty={\mathcal W}_K=1$
and the most deep and mysterious invariant about $p$-rationality is 
the normalized $p$-adic regulator which then becomes, for $K$ real and 
$p$ unramified, ${\mathcal R}_K = \frac{1}{p^{r_1+r_2-1}} \,R_K$
($R_K$ being the usual $p$-adic regulator \cite[\S\,5.5]{W}).

\smallskip
In terms of $p$-rationality, the real case may be written in a conjectural 
way as follows (for more information and generalizations to Jaulent's conjecture
replacing $E_K$ by the $G$-module generated by an
algebraic number $\eta$, see \cite[Theorem 1.1 \& Heuristic 7.4]{Gr3}):

\begin{conjecture}
{\it Let $K/\Q$ be a real Galois extension of Galois group $G$ and
let $\eta \in E_K$ be a Minkowski unit (i.e., a unit generating a 
sub-$G$-module of finite index of $E_K$).
The probability of ${\mathcal R}_K  \equiv 0 \pmod p$ is at most:
$$\hbox{$\ds \frac{1}{p^{ {\rm log}_2(p)/ {\rm log}(c_0(\eta))-O(1) }}\,,$
for $p\to \infty$, }$$
where $c_0(\eta) = {\rm max}_{\sigma \in G} (\vert \eta^\sigma \vert)$, 
and ${\rm log}_2={\rm log} \circ {\rm log}$.

\smallskip  
Under the principle of Borel--Cantelli, the number of primes $p$ such that 
$K$ is non-$p$-rational is finite. }
\end{conjecture}

\subsection*{Acknowledgments} My thanks to Christian Maire for many 
exchanges and discussions about $p$-rationality, and to Jean-Fran\c cois 
Jaulent about the Bertrandias--Payan module for $p=2$.

\end{document}